\documentclass[a4paper,10pt]{amsart}

\usepackage[english]{babel}
\usepackage[utf8]{inputenc}

\usepackage{amssymb}
\usepackage{amsmath}
\usepackage{amsthm}
\usepackage{bbm}

\usepackage{enumerate}
\usepackage{tikz}
\usepackage{marginnote}

\newcommand{\bbC}{\mathbb{C}}

\newcommand{\bbN}{\mathbb{N}}

\newcommand{\bbR}{\mathbb{R}}

\newcommand{\bbZ}{\mathbb{Z}}

\newcommand{\calL}{\mathcal{L}}

\newcommand{\calU}{\mathcal{U}}

\DeclareMathOperator{\id}{id} 
\DeclareMathOperator{\fix}{fix} 
\newcommand{\norm}[1]{\left\lVert #1 \right\rVert} 
\newcommand{\modulus}[1]{\left\lvert #1 \right\rvert} 

% ...

\theoremstyle{definition}
\newtheorem{definition}{Definition}[section]

\newtheorem{remarks}[definition]{Remarks}

\newtheorem{example}[definition]{Example}
\newtheorem{examples}[definition]{Examples}
\newtheorem{open_problem}[definition]{Open Problem}

\theoremstyle{plain}

\newtheorem{theorem}[definition]{Theorem}
\newtheorem{corollary}[definition]{Corollary}

\numberwithin{equation}{section}

\begin{document}

\title[A note on the fixed space of positive contractions]{A note on the fixed space of positive contractions}
\author{Jochen Gl\"uck}
\email{jochen.glueck@uni-passau.de}
\address{Jochen Gl\"uck, Faculty of Computer Science and Mathematics, University of Passau, Innstr.\ 33, 94032 Passau, Germany}
\date{\today}
\begin{abstract}
	We prove that, in a large class of Banach lattices, the fixed space of each commuting family of positive linear contractions is a lattice subspace. 
	As consequences, new cyclicity results for the peripheral point spectra of positive operators and semigroups are derived; we also pose an open problem that naturally occurs in this context.
	Finally, a variety of counterexamples is presented to point out some limits of our results.
\end{abstract}

\maketitle

\section{Introduction}

\subsection*{The fixed space of positive operators}

Consider a positive linear operator $T$ on a Banach lattice $E$.
There are various sufficient conditions which ensure that the \emph{fixed space} $\fix T := \ker(\id-T)$ is a lattice in its own right. 
For instance, if $T$ is \emph{contractive}, i.e., $\norm{T} \le 1$, and the norm on $E$ is \emph{strictly monotone} in the sense that $\norm{f} < \norm{g}$ whenever $0 \le f \le g$ but $f \not= g$, then it is easy to check that $\fix T$ is even a \emph{sublattice} of $E$:
for $f \in \fix T$ one has $\modulus{f} = \modulus{Tf} \le T \modulus{f}$ and hence, by the contractivity of $T$ and the strict monotonicity of the norm, $\modulus{f} = T \modulus{f}$.

If $T$ is only power bounded rather than contractive, but $E$ is a so-called \emph{KB}-space (i.e., every increasing norm bounded net in $E$ is norm convergent), then $\fix T$ might not be a sublattice of $E$, but it is still a \emph{lattice subspace} -- i.e., a lattice with respect to the order inherited from $E$ (though not necessarily with the same lattice operations as $E$).
Indeed, for $f \in \fix T$, it follows from $\modulus{f} \le T \modulus{f}$ that the sequence $(T^n\modulus{f})_{n \in \bbN_0}$ is increasing, and thus convergent to a point $h \in \fix T$;
it is straightforward to check that $h$ is the smallest upper bound of $\pm f$ within $\fix T$, so $\fix T$ is indeed a lattice.

For various further sets of conditions that ensure that the fixed space of a positive operator is a lattice we refer, for instance, to \cite[Propositions~III.8.4 and III.11.5, and the proof of Proposition~III.8.11]{Schaefer1974}, to the main result of \cite{Christianson1989}, and to \cite[Section~2]{Glueck2016} and \cite[Proposition 3.11]{Gerlach2019}.

The fact that the fixed space is a lattice subspace, or even a sublattice, is instrumental in the spectral theory of positive operators and semigroups; for examples of this, we refer to \cite[Corollary~C-III-4.3]{Arendt1986}, \cite[Theorem~3.2 and its consequences]{Glueck2016}, and \cite[Theorem~5.3]{Gerlach2019}.
We will also give two examples of such applications in Corollaries~\ref{cor:pps-single-operator} and~\ref{cor:pps-semigroup} below.

\subsection*{Contributions and organization of the article}

In this article we give a new sufficient criterion for the fixed space of a positive linear contraction $T$ to be lattice subspace of $E$.
In contrast to the situation above, we do not require $E$ to be a KB-space nor to have strictly monotone norm; instead, we assume that $E$ is \emph{monotonically complete} and has the \emph{Fatou property}. 
Those two properties are not only satisfied by every KB-space, but also by every dual Banach lattice, and by the space of continuous functions on any Stonian compact Hausdorff space (Examples~\ref{exas:monotonic-completeness-and-fatou}).

The article is organised as follows: 
In the remainder of the introduction we recall a bit of notation and terminology. 
Section~\ref{sec:geometry} contains a brief reminder of Banach lattices which have the Fatou property and are monotonically complete. 
Our main result, Theorem~\ref{thm:main}, is stated and proved in Section~\ref{sec:fixed-space}; moreover, we give two brief applications to the cyclicity of the peripheral point spectrum in this section.
The final Section~\ref{sec:examples} contains a number of counterexamples.

\subsection*{A bit of terminology}

We use the conventions $\bbN := \{1,2,\dots\}$ and $\bbN_0 := \bbN \cup \{0\}$.

Throughout, we freely use the language and theory of real and complex Banach lattices, as presented for instance in the standard monographs \cite{Aliprantis2006, Meyer-Nieberg1991, Schaefer1974, Zaanen1983}.

Let $E$ be a real Banach lattice and let $F \subseteq E$ be a vector subspace.
The space $F$ is called a \emph{sublattice} of $E$ if $\modulus{f} \in F$ for each $f \in F$;
it is called a \emph{lattice subspace} of $E$ if it is, with respect to the order inherited from $E$, a vector lattice in its own right (but possibly with different lattice operations).
It is easy to see that every sublattice is also a lattice subspace, but the converse is not true in general.

Now let $E$ be a complex Banach lattice. 
Then $E$ is the \emph{complexification} of a real Banach lattice $E_\bbR$; the latter is called the \emph{real part} of $E$.
Let $F \subseteq E$ be a vector subspace. 
We say that $F$ is \emph{invariant under complex conjugation} if, for all $f \in F$, the complex conjugate vector $\overline{f}$ is also in $F$. 
The space $F$ is called a \emph{sublattice} of $F$ if it is invariant under complex conjugation and $\modulus{f} \in F$ for each $f \in F$. 
If $F$ is closed, this is equivalent to $F$ being invariant under complex conjugation and its \emph{real part} $F_\bbR := E \cap E_\bbR$ being a sublattice of $E_\bbR$ (see for instance \cite[Remark~C.5.3]{GlueckThesis}).
Finally, we call $F$ a \emph{lattice subspace} if it is invariant under complex conjugation and the real part $F_\bbR$ is a lattice subspace of $E_\bbR$.

For more information about lattice subspaces we refer for instance to \cite{Miyajima1983} (where the term \emph{quasi-sublattice} is used instead) and \cite{Polyrakis1999}.

To keep the terminology simple, we will refer to a real Banach lattice as its own real part, such that the ``real part of a Banach lattice $E$'' is always well-defined, no matter whether the underlying scalar field of $E$ is real or complex.
A vector $f$ in a real or complex Banach lattice is called \emph{real} if it is an element of the real part of $E$.

Throughout, the underlying scalar field of a given function or sequence space is always allowed to be real or complex, unless specified otherwise.

\section{The Fatou property and monotone completeness}
\label{sec:geometry}

Let $E$ be a real or complex Banach lattice. Then $E$ is said to have the \emph{Fatou property} \cite[Definition~2.4.18(i)]{Meyer-Nieberg1991} if every increasing net $(x_j)_{j \in J} \subset E_+$ which posses a supremum $x \in E$ satisfies
\begin{align*}
	\|x\| = \sup_{j \in J} \|x_j\|.
\end{align*}
Obviously, a complex Banach lattice has the Fatou property if and only if the same is true for its real part.
Clearly, every Banach lattice with order continuous norm has the Fatou property, and so has every space of continuous scalar-valued functions on a compact Hausdorff space, endowed with the sup norm.

The Banach lattice $E$ is said to be \emph{monotonically complete} if every norm bounded increasing net in $E_+$ has a supremum \cite[Definition~2.4.18(iii)]{Meyer-Nieberg1991}. Equivalently, every norm bounded increasing net in the real part of $E$ has a supremum.

\begin{remarks}
	\begin{enumerate}[(a)]
		\item 
		If $E$ is monotonically complete, then it is also \emph{Dedekind complete} (i.e., every non-empty order bounded set in $E$ has a supremum); see \cite[Proposition~2.4.19(i)]{Meyer-Nieberg1991}.
		
		The converse implication does not holds, though:
		The space $c_0$ of scalar-valued sequences that converge to $0$ is Dedekind comoplete but not monotonically complete.
		
		\item 
		Monotonic completeness is logically independent from having the Fatou property:
		
		The space $c_0$ has the Fatou property because it has order continuous norm; however, it is not monotonically complete. 
		To obtain another class of examples, one easily checks that for every compact Hausdorff space $K$ the space $C(K)$ of continuous scalar-valued functions on $K$ has the Fatou property; however, if $K$ is not Stonian, then $C(K)$ is not order complete and thus, it cannot be monotonically complete. 
		
		Conversely, an example of a Banach lattice which is monotonically complete, but does not have the Fatou property can be found in \cite[Example (ii) on p.\,98]{Meyer-Nieberg1991} (compare however \cite[Proposition~2.4.19(i)]{Meyer-Nieberg1991}).
	\end{enumerate}
\end{remarks}

Our main result in the next section holds on Banach lattices which are both monotonically complete and have the Fatou property. In the subsequent examples we list some classes of Banach lattices which have both these properties.
For all these spaces the Fatou property and the monotonic completeness are very easy to check, so we omit a detailed proof.

\begin{examples} 
	\label{exas:monotonic-completeness-and-fatou}
	Let $E$ be a (real or complex) Banach lattice. Each of the following properties is sufficient for $E$ to be monotonically complete and have the Fatou property:
	\begin{itemize}
		\item[(a)] We have $E = C(K)$ for a compact Hausdorff space $K$, and $E$ is Dedekind complete (i.e., $K$ is Stonian).
		
		\item[(b)] The space $E$ is a KB-space.
		
		\item[(c)] The space $E$ has a pre-dual Banach lattice, i.e., there is a Banach lattice $F$ such that $E = F'$.
	\end{itemize}
\end{examples}

\section{The fixed space of a family of positive contractions} 
\label{sec:fixed-space}

The following is our main result.

\begin{theorem}
	\label{thm:main}
	Let $E$ be a real or complex Banach lattice which is monotonically complete and has the Fatou property.
	Let $T = (T_i)_{i \in I}$ be a family of positive and contractive (i.e., $\norm{T_i} \le 1$) linear operators on $E$ which mutually commute. 
	Then the \emph{fixed space}
	\begin{align*}
		F := \fix T := \{x \in E: \, T_i x = x \text{ for all } i \in I\}
	\end{align*}
	of $T$ is a closed vector subspace of $E$ with the following properties:
	\begin{enumerate}[\upshape (a)]
		\item
		The space $F$ is a lattice subspace of $E$.

		\item
		The space $F$ is a Banach lattice with respect to the norm induced by $E$.
		
		\item
		The Banach lattice $F$ is monotonically complete and has the Fatou property. 
		In particular, $F$ is Dedekind complete.

		\item 
		Let $G \subseteq F$ be a set of real vectors which has a supremum $g_E$ in $E$.
		Then $G$ has also a supremum $g_F$ in $F$, and we have $g_E \le g_F$.
		If $g_E \ge 0$, then $\|g_F\| = \|g_E\|$.
	\end{enumerate}
\end{theorem}

\begin{proof}
	We first prove (d). Afterwards, we will see that the remaining assertions follow easily from~(d).
	
	(d)
	Let $G \subseteq F$ and $g_E$ be as in the claim. Denote the set of all upper bounds of $G$ within $F$ by $U$ (but note that we do not know a priori that $U$ is non-empty).
	We consider the set $H$ of all real vectors $h \in E$ that satisfy the following properties:
	\begin{enumerate}[(1)]
		\item 
		The vector $h$ is bounded below by $G$ and bounded above by $U$.
		
		\item 
		The vector $h$ is a \emph{super fixed vector} of $T$, i.e., $T_i h \ge h$ for all $i \in I$.
		
		\item 
		We have $\norm{h} \le \norm{g_E}$.
	\end{enumerate}
	The set $H$ is non-empty since it contains $g_E$.
	Let us show that $H$ is invariant under each operator $T_i$. So fix $h \in H$ and $i \in I$.
	Then $T_i h$ satisfies~(1) since $T_i$ is positive and $G$ and $U$ consist of fixed vectors of $T_i$. The vector $T_i h$ also satisfies~(2) since $T_i$ is positive and all operators in the family $T$ commute. Finally, $T_i h$ satisfies~(3) since $\norm{T_i h} \le \norm{h} \le \norm{g_E}$, as $T_i$ is contractive.
	
	Next we show that $H$ has a maximal element. Let $C \subseteq H$ be a non-empty chain in $H$; by Zorn's lemma it suffices to prove that $C$ has an upper bound in $H$.
	Since the norms of all vectors in $C$ are bounded by $\norm{g_E}$, it follows from the monotonic completeness of $E$ that $C$ has a supremum $c$ in $E$. 
	Moreover, $\norm{c} \le \norm{g_E}$ since $E$ has the Fatou property.
	It is clear that $c$ is also bounded below by $G$ and above by $U$.
	Finally, as $c$ is a supremum of super fixed vectors of $T$, it is itself of super fixed vector, too. 
	Hence, $c \in H$, so $c$ is indeed an upper bounded of $C$ within $H$.
	
	Now, let $g_F$ be a maximal element of $H$.
	For each index $i \in I$, the vector $T_i g_F$ is in $H$ due to the $T_i$-invariance of $H$, and we have $T_i g_F \ge g_F$ by property~(2) in the definition of $H$;
	hence, $T_i g_F = g_F$ due to the maximality of $g_F$.	
	Thus, $g_F$ is an element of $F$. 
	It is also an upper bounded of $G$ and smaller than each further upper bound of $G$ within $F$; this proves that $g_F$ is the supremum of $G$ within $F$.
	
	Obviously, $g_E \le g_F$. Since $g_F$ is an element of $H$, it satisfies $\norm{g_F} \le \norm{g_E}$. If $g_E$ is positive, then the inequality $g_E \le g_F$ implies that we also have $\norm{g_E} \le \norm{g_F}$, so $\norm{g_E} = \norm{g_F}$.
	
	(a) 
	Let $f \in F$ be a real vector. 
	Then it follows from~(d) that $\{f, -f\}$ has a supremum $\modulus{f}_F$ in $F$. 
	Moreover, if the scalar field is complex, then $F$ is invariant under complex conjugation since all operators $T_i$ commute with complex conjugation (this is a consequence of the positivity of the $T_i$).
	Hence, $F$ is a lattice subspace of $E$, both in the real and in the complex case.
	
	(b)
	First, let $f \in F$ be a real vector. 
	The the modulus $\modulus{f}_F$ of $f$ within $F$, which is the supremum of $f$ and $-f$ within $F$, has the same norm as $f$ according to~(d).
	Hence, the real part of $F$ is a real Banach lattice with respect to the norm induced by $E$.
	
	We still need to show that, if the scalar field is complex, the norm induced by $E$ coincides with the Banach lattice complexification of the norm on the real part of $F$.
	So let $f = f_1 + if_2 \in F$ for real vectors $f_1, f_2 \in F$. 
	Let $\modulus{f}_E$ and $\modulus{f}_F$ denote the moduli of $f$ in $E$ and $F$, respectively.
	These are the suprema of the set
	\begin{align*}
		\{f_1 \cos \theta + f_2 \sin \theta : \, \theta \in [0,2\pi) \}
	\end{align*}
	in $E$ and $F$, respectively.
	Hence, it follows from~(d) that $\norm{\modulus{f}_F} = \norm{\modulus{f}_E} = \norm{f}$.
	
	(c) 
	Since $E$ is monotonically complete, it follows from~(d) that $F$ is monotonically complete, too.
	In particular, $F$ is Dedekind complete. 
	Using the fact that $E$ has the Fatou property, and again~(d), we see that $F$ has the Fatou property, too.
\end{proof}

\begin{remarks}
	\label{rems:main-theorem}
	\begin{enumerate}[(a)]
		\item 
		If, in the situation of Theorem~\ref{thm:main}, $E$ is an AM-space, then it follows from~(d) that $F$ is an AM-space, too. 
		
		\item 
		Theorem~\ref{thm:main} is thus a generalisation of \cite[Theorem~2.1 and Corollary~2.2]{Glueck2016}, where the same result (though with slightly less detailed assertions) was shown for the special case of single Markov operators on Dedekind complete $C(K)$-spaces.
		The main idea of the proof -- invoking Zorn's lemma for an appropriately chosen set of super fixed points -- was similar there.
		
		\item 
		Under the assumptions of Theorem~\ref{thm:main}, the lattice subspace $F$ need not be a sublattice, in general. 
		A counterexample can be found in \cite[Example~2.3(a)]{Glueck2016}.
		
		\item 
		Let us consider the situation in Theorem~\ref{thm:main}(d), but assume now that the operator family $(T_i)_{i \in I}$ consists of a single operator $T$ only.
		Instead of using Zorn's lemma in the proof of~(d), one can also argue by transfinite induction. 
		It is instructive to discuss this in a bit more detail:
		
		We have $Tg_E \ge g_E$, and by iterating this inequality we obtain that the sequence $(T^n g_E)_{n \in \bbN_0}$ is increasing; since $T$ is contractive, the sequence is also norm bounded by the norm of $g_E$.
		Hence, since $E$ is monotonically complete, the sequence $(T^n g_E)_{n \in \bbN_0}$ has a supremum $g_1$ in $E$, whose norm is dominated by $\norm{g_E}$ since $E$ has the Fatou property.
		If $T$ were order continuous, $g_1$ would be a fixed point of $T$ and it would readily follow that $g_1$ is the supremum of $G$ within $F$.
		Since, however, $T$ is not assumed to be order continuous, $g_1$ will only be a super fixed point of $T$, in general.
		Hence, we now consider the sequence $(T^n g_1)_{n \in \bbN_0}$, which is again increasing and norm bounded by $\norm{g_E}$; thus, it has a supremum $g_2$.
		By iterating this procedure transfinitely often, we arrive at a fixed point of $T$ which turns out to be the supremum of $G$ within the fixed space $F$.
		(Note that the entire procedure terminates due to cardinality reasons.)
		
		This transfinite argument also gives a very clear idea of why contractivity of $T$ (rather than only power boundedness) and the Fatou property of $E$ are important: 
		we need these assumptions to ensure that the chains that occur in the transfinite induction remain norm bounded.
		The discussion after Example~\ref{ex:power-bbd-is-not-sufficient-for-cyclic-pps} below demonstrates how this fails if the operator is merely power bounded.
		
		\item 
		Almost the same argument as in the proof of Theorem~\ref{thm:main}(d) shows the following result (under the assumptions of the theorem):
		
		If $g \in E$ is a real vector which is a super fixed point of the family $T = (T_i)_{i \in I}$, then there exists a smallest fixed vector $f \in F$ which satisfies $g \le f$; moreover, if $g \ge 0$, then $\norm{f} = \norm{g}$.
	\end{enumerate}
\end{remarks}

Let us briefly discuss a few applications of Theorem~\ref{thm:main} to the spectral theory of positive operators and operator semigroups.

\begin{corollary}
	\label{cor:pps-single-operator}
	Let $E$ be a complex Banach lattice which is monotonically complete and has the Fatou property.
	Let $T$ be a positive and contractive linear operator on $E$.
	If $\lambda \in \bbC$ is a root of unity and an eigenvalue of $T$, then so is $\lambda^k$ for each $k \in \bbZ$; even more, one has the dimension estimate
	\begin{align*}
		\dim \ker(\lambda - T) \le \dim \ker(\lambda^k - T)
	\end{align*}
	for the corresponding eigenspaces (where we do not distinguish between different infinite cardinalities, i.e., the dimension is understood to be either an integer or $\infty$).
\end{corollary}

Results of the type ``If a complex number $\lambda$ of modulus $1$ is an eigenvalue of the positive operator $T$, then so is $\lambda^k$ for each integer $k$'' are called \emph{cyclicity results} in Perron--Frobenius theory. 
Such results results are only true under appropriate assumptions on the operator and the underlying space (see for instance \cite[Examples~5.4 and~6.2]{Glueck2016} for two counterexamples which show that cyclicity fails, in general).
Numerous sets of assumptions are known which imply such cyclicity results; we refer for instance to \cite[Section~3]{Scheffold1971}, \cite[Section~V.4]{Schaefer1974}, \cite[Section~4]{Grobler1995}, and \cite[Sections~5 and~6]{Glueck2016}.
Recent applications of cyclicity results in ergodic theory can be found in \cite[Section~3]{Cohen2018} and \cite[Proposition~5.2]{Eisner2018}.

We point out that Corollary~\ref{cor:pps-single-operator} is already known in the special case where $T$ is a Markov operator on a Dedekind complete $C(K)$-space \cite[Proposition~6.4]{Glueck2016}.
Corollary~\ref{cor:pps-single-operator} in the general version above can be proved by almost the same argument -- one merely needs to replace the reference to \cite[Corollary~2.2]{Glueck2016} in the proof of \cite[Proposition~6.4]{Glueck2016} with a reference to Theorem~\ref{thm:main} of the present article.
Still, we prefer to include the entire argument since it demonstrates nicely how Theorem~\ref{thm:main} is used.

\begin{proof}[Proof of Corollary~\ref{cor:pps-single-operator}]
	Choose an integer $n \ge 1$ such that $\lambda^n = 1$, and consider the space
	\begin{align*}
		F := \fix(T^n).
	\end{align*}
	This is a closed an $T$-invariant vector subspace of $E$, and it contains the eigenspace $\ker(\lambda - T)$; hence, $\ker(\lambda - T) = \ker(\lambda - T|_F)$.
	In particular, $\lambda$ is an eigenvalue of the restriction $T|_F$.
	Moreover, Theorem~\ref{thm:main} shows that $F$ is a Banach lattice with respect to the order and norm inherited from $E$. Since we have
	\begin{align*}
		T|_F (T|_F)^{n-1} = (T|_F)^{n-1} T|_F = (T^n)|_F = \id_F,
	\end{align*}
	it follows that $T|_F$ is a bijective operator on $F$, and its inverse $(T|_F)^{n-1}$ is positive.
	Hence, $T|_F$ is even a lattice isomorphism on the complex Banach lattice $F$.
	
	But for lattice homomorphisms, the claimed dimension estimate is known. 
	More precisely, we have for each integer $k$
	\begin{align*}
		\dim \ker(\lambda - T) = \dim \ker(\lambda - T|_F) \le \dim \ker (\lambda^k - T|_F) \le \dim \ker(\lambda^k - T),
	\end{align*}
	where the estimate in the middle is a property of lattice homomorphisms that can be found in \cite[Proposition~3.1]{Glueck2016}.
	So the claimed dimension estimate holds; in particular, $\lambda^k$ is an eigenvalue of $T$.
\end{proof}

The proof of Corollary~\ref{cor:pps-single-operator} relies heavily on the assumption that $\lambda$ be a root of unity. 
Thus, the following question naturally arises:

\begin{open_problem}
	Does the assertion of Corollary~\ref{cor:pps-single-operator} remain true if $\lambda$ is not a root of unity, but more generally a complex number of modulus $\modulus{\lambda} = 1$?
\end{open_problem}

On a similar note, one can prove a cyclicity result about the peripheral point spectrum of semigroup generators.
To understand the following corollary, familiarity with the theory of $C_0$-semigroups is required. 
An excellent treatment of this theory can, for instance, be found in \cite{Engel2000}.

\begin{corollary}
	\label{cor:pps-semigroup}
	Let $E$ be a complex Banach lattice which is monotonically complete and has the Fatou property.
	Let $(T_t)_{t \in [0,\infty)}$ be a positive and contractive $C_0$-semigroup on $E$ with generator $A$.
	If $i\beta \in i\bbR$ is an eigenvalue of $A$, then so is $ik\beta$ for each integer $k \in \bbZ$.	 
\end{corollary}
\begin{proof}
	As explained in the proof of \cite[Corollary~C-III-4.3]{Arendt1986}, it suffices to prove that, for each time $t$, the fixed space of $T_t$ is a Banach lattice. 
	In our situation, this is the case due to Theorem~\ref{thm:main}.
\end{proof}

Similar cyclicity results for the peripheral point spectrum of generators of positive semigroups can be found, under different sets of assumptions, in \cite[Corollary~C-III-4.3]{Arendt1986}.
On the other hand, for various counterexamples which show that such a result does not hold without any assumptions on the semigroup and the underlying space, we refer to \cite[Examples~B-III-2-13 and~C-III-4.4]{Arendt1986} and \cite[Example~8.2]{Glueck2016}.

Most likely, a similar dimension estimate as in Corollary~\ref{cor:pps-single-operator} can also be shown in the semigroup case.
However, we refrain from discussing this in detail here since (i) for the corresponding dimension estimate for semigroups of lattice homomorphisms that would be needed in the proof, no such convenient estimate as \cite[Proposition~3.1]{Glueck2016} seems to be available, and (ii) we prefer not to divert too much from our focus on the fixed space of positive operators.

\section{Counterexamples}
\label{sec:examples}

Now we give a few counterexamples to demonstrate some limits of our results and arguments from the previous section.

The sequence space $c_0$ has order continuous norm, so it has the Fatou property and is Dedekind complete; but it is not monotonically complete.
So Theorem~\ref{thm:main} is not applicable on this space -- and indeed, the assertions of the theorem do not hold there.
This can be concluded from \cite[Example~6.2 and Remark~6.3]{Glueck2016}, but the example there is more complicated than needed in the present paper (since this example was constructed as a counterexample for a different question).
Here is a simpler version of this example that suffices for our purposes.

\begin{example} \label{ex:contraction-on-c_0}
	Let $E = c_0(\bbN_0 \cup \{-1,-2\})$ denote the space of scalar-valued sequences, indexed over $\bbN_0 \cup \{-1,-2\}$, which converge to $0$.
	As is common, we endow this space with the sup norm. There is a positive, contractive operator $T \in \calL(E)$ whose fixed space $F := \ker(1-T)$ is not a lattice subspace of $E$.
	
	Indeed, define $T$
	\begin{align*}
		(Tf)(k) =
		\begin{cases}
			f(k-1) \quad & \text{if } k \in \bbN \\
			\frac{1}{2}\big(f(-1)+f(-2)\big) \quad & \text{if } k = 0 \\
			f(k) \quad & \text{if } k \in \{-1,-2\}
		\end{cases}
	\end{align*}
	for all $f \in E$ and all $k \in \bbN$. Clearly, $T$ is a positive, contractive linear operator on $E$ and one immediately checks that the fixed space of $T$ is given by
	\begin{align*}
		F = \{f \in E| \, f(k) = 0 \text{ for all } k \in \bbN_0 \text{ and } f(-1) = -f(-2) \}.
	\end{align*}
	In particular, $F$ is non-zero, but it does not contain any positive vectors except $0$. Hence, $F$ cannot be a lattice subspace of $E$. 
\end{example}

In the preceding example, we argued that $F$ is not a lattice subspace since the cone in $F$ is too small (in this particular example, it is even $\{0\}$).
One might thus wonder whether at least the linear span of the positive fixed vectors of $T$ always form a lattice. 
On the space $c_0$, and more generally on Banach lattices with order continuous norm, the answer is positive. This was proved in the main result of \cite{Christianson1989} (see also \cite{Christianson1995}). 
It is worthwhile pointing out that the span of the positive fixed vectors need not be closed, in general (see \cite[Example~4]{Christianson1989}); however, as also shown in the main theorem in \cite{Christianson1989}, it can always be endowed with a stronger norm which renders it a Banach lattice.

In the proof of Theorem~\ref{thm:main}(d) we obtained the supremum of a subset $G$ of $\fix T$ by an application of Zorn's lemma; 
as explained in Remark~\ref{rems:main-theorem}(d), if $T$ is a single operator, this can also be done by iterating $T$ transfinitely often on the supremum of $G$ within $E$. One might wonder whether this transfinite construction is really necessary, or if the stationary point is in fact always reached at the first infinite ordinal number. Let us demonstrate by a simple example that one indeed needs limit steps in the transfinite induction, in general.

\begin{example} 
	\label{ex:necessity-of-transfinite-construction}
	(a) 
	We begin with a preliminary example which will be needed in our actual example (b) below. Consider the operator $S$ on $\bbR^3$ whose representation matrix with respect to the canonical basis is given by
	\begin{align*}
		\begin{pmatrix}
			1 & 0 & 0 \\
			\frac{1}{3} & \frac{1}{3} & \frac{1}{3} \\
			0 & 0 & 1
		\end{pmatrix}.
	\end{align*}
	This operator is taken from \cite[Example~2.13(a)]{Glueck2016}; it is a Markov operator on $\bbR^3$ whose fixed space $\ker(1-S)$ is spanned by the vectors $\mathbbm{1}_3 := (1,1,1)$ and $\hat f :=(1,0,-1)$. Note that $\ker(1-S)$ is a lattice subspace of $\bbR^3$ by Theorem~\ref{thm:main}(a) (or by \cite[Theorem~2.1]{Glueck2016}), however it is not a sublattice of $\bbR^3$ as was pointed out in \cite[Example~2.13(a)]{Glueck2016} (and as is easy to see). 

	(b) 
	Let $E = \bbR^3 \times \ell^\infty \times \ell^\infty$ be endowed with the canonical order and the supremum norm; we use $\bbN = \{1,2,\dots\}$ as index set for $\ell^\infty$, i.e., vectors $g \in \ell^\infty$ are given by $g = (g_n)_{n \in \bbN} = (g_1, g_2, \dots)$. 
	The space $E$ is a dual Banach lattice (and also isometrically lattice isomorphic to $C(K)$ for a Stonian compact Hausdorff space $K$); we can therefore apply Theorem~\ref{thm:main} on this space. 
	Now, fix a free ultrafilter $\calU$ on $\bbN$ and define an operator $T \in \calL(E)$ in the following way: For each $(f,g,h) \in E$ we define $T(f,g,h) = (f',g',h')$, where
	\begin{align*}
		f' = Sf, \qquad g' = (\frac{f_1+f_3}{2},g_1,g_2,g_3,...), \qquad h' = (\lim_\calU g, h_1,h_2,h_3,...);
	\end{align*}
	here, $S \in \calL(\bbR^3)$ is the operator from example (a). It is easy to check that $T$ is positive and that $r(T) = \|T\| = 1$. Moreover, the fixed space $F := \ker(1-T)$ is easily seen to be spanned by the two vectors $(\mathbbm{1}_3, \mathbbm{1}_\bbN, \mathbbm{1}_\bbN) \in E$ and $(\hat f, 0, 0) \in E$, where $\mathbbm{1}_3$ and $\hat f$ are defined as in example (a).
	
	Let us now demonstrate how the transfinite argument mentioned in Remark~\ref{rems:main-theorem}(d) yields the supremum of $G = \{(\hat f, 0, 0), (-\hat f,0,0)\}$ in $F$.
	To this end, let $g_E = (|\hat f|, 0, 0)$ be the supremum of $G$ in $E$.
	By iterating $T$ on $g_E$ and taking the supremum we obtain the vector $g_1 := \sup_{n \in \bbN_0} g_E = (\mathbbm{1}_3, \mathbbm{1}_\bbN, 0)$.
	This is only a super fixed point of $T$, but not a fixed point -- so we continue be iterating $T$ on $g_1$.
	Thus, we obtain the supremum $g_2 := \sup_{n \in \bbN_0} T^n g_1 = (\mathbbm{1}_3,\mathbbm{1}_\bbN, \mathbbm{1}_\bbN)$, which is indeed a fixed point of $T$. 
	Hence, the supremum of $G$ within $F$ (which is the modulus of $(\hat f, 0, 0)$ within $F$) is the vector $g_2 = (\mathbbm{1}_3,\mathbbm{1}_\bbN, \mathbbm{1}_\bbN)$.
\end{example}

It is natural to ask whether the assertions of Theorem~\ref{thm:main} and Corollary~\ref{cor:pps-single-operator} remain true of we replace the condition $\|T\| = 1$ by, say, power-boundedness of $T$. The following example shows that the answer is negative.

\begin{example} 
	\label{ex:power-bbd-is-not-sufficient-for-cyclic-pps}
	There is a Dedekind complete $C(K)$-space $E$ and a positive operator $T \in \calL(E)$ such the following assertions hold:
	\begin{enumerate}[(a)]
		\item 
		The operator $T$ is power bounded and has spectral radius $1$.
		
		\item 
		The number $-1$ is an eigenvalue of $T$, but $1$ is not.
		
		\item 
		The fixed space of $T^2$ is not a lattice subspace of $E$.
	\end{enumerate}
	Indeed, consider the space $E = \ell^\infty(\{-2,-1\} \, \cup \, (\bbN_0 \times \bbN_0))$, where $\{-2,-1\} \, \cup \, (\bbN_0 \times \bbN_0)$ denotes the index set of vectors in $E$. 
	The space $E$ is isometrically lattice isomorphic to an order complete $C(K)$-space. Fix a free ultra filter $\calU$ on $\bbN_0$. We define an operator $T \in \calL(E)$ in the following way: for $f \in E$, let
	\begin{align*}
		(Tf)(-2) & = f(-1), \\
		(Tf)(-1) & = f(-2), \\
		(Tf)(k,j) & =
		\begin{cases}
			\frac{1}{2}(f(-2) + f(-1)) \quad & \text{if } k=j=0 \\
			2\lim_{m \to \calU} f(k-1,m) \quad & \text{if } k \ge 1 \text{ and } j=0 \\
			f(k,j-1) \quad & \text{if } j \ge 1.
		\end{cases}
	\end{align*}
	This is clearly a positive linear operator on $E$ with $\|T\| = 2$. 
	
	Let us show that $T$ satisfies the claimed properties (a)--(c):
	
	(a) It is easy to see that we have $\|T^n\| = 2$ for every $n \in \bbN$. Hence, $T$ is power bounded and has spectral radius $1$. 
	
	(b) Let $f \in E$ be given by $f(-2) = 1$, $f(-1) = -1$ and $f(k,j) = 0$ for all $(k,j) \in \bbN_0 \times \bbN_0$. Then we have $Tf = -f$ and thus, $-1$ is an eigenvalue of $T$.
	
	On the other hand, assume for a contradiction that $1$ is an eigenvalue of $T$ and that $g \in E$ is a corresponding eigenfunction. If $g(-2) = 0$, then we can readily check that $g$ must be constantly $0$, which is a contradiction. Hence, $g(-2) \not= 0$ and we may therefore assume that $g(-2) = 1$. Thus, $g(-1) = 1$, too and this yields by a (double) induction that $g(k,j) = 2^k$ for all $(k,j) \in \bbN_0 \times \bbN_0$. This is a contradiction since $g$ is bounded. Hence, $1$ is not an eigenvalue of $T$.
	
	(c) It is not difficult to see that $\ker(1-T^2) = \ker(-1-T) \oplus \ker(1-T)$ (this is true for every linear operator). Hence, it follows from (b) that $\ker(1-T^2) = \ker(-1-T)$ and the latter space is easily seen to be spanned by the vector $f$ given in (b). Thus, $\ker(1-T^2)$ is non-zero, but it does not contain any non-zero positive vectors. Therefore, it cannot be a lattice subspace of $E$.
\end{example}

It is instructive to observe how the transfinite argument from Remark~\ref{rems:main-theorem}(d) fails for the operator $T^2$ in the above example: for example, let $f\in \ker(-1-T) = \ker(1-T^2)$ be the vector given in the proof of property~(b) of the example. 
If we wanted to follow the argument from Remark~\ref{rems:main-theorem}(d) in order to construct a supremum of $\{-f,f\}$ in the fixed space $F$ of $T^2$, then we had to iterate the operator $T^2$ on $|f|$, take the supremum of $(T^{2n}\modulus{f})_{n \in \bbN_0}$, iterate $T^2$ on this supremum again, and so on.
For each $k \in \bbN$, let $g_k$ denote the vector that one thus obtains by taking the supremum for the $k$-th time. Then it is easy to see that $g_k$ doubles its norm for each $k$, so the sequence $(g_k)_{k \in \bbN}$ is unbounded and thus does not have a supremum. Hence, the construction cannot be continued.

The situation changes of we strengthen the assumptions on the Banach lattice $E$. 
For instance, if $E$ is a KB-space, then it is easy to see that the fixed space of every power bounded positive operator $T$ on $E$ is a lattice subspace (though one might need to switch to an equivalent norm in order to render it even a Banach lattice): this follows from the argument given in the first part of the introduction.
However, we find it worthwhile to point out that, even in finite dimensions, one cannot expect the fixed space to be a lattice subspace if one also drops the condition of power boundedness.
This can be seen in the following simple example.

\begin{example}
	There is a (non-power bounded) positive linear operator $T$ on $E = \bbR^3$ with spectral radius $1$ such that the fixed space $\ker(1-T)$ is not a lattice subspace of $\bbR^3$: Indeed, let 
	\begin{align*}
		T = 
		\begin{pmatrix}
			1 & 0 & 0 \\
			1 & 1 & 1 \\
			0 & 0 & 1
		\end{pmatrix}.
	\end{align*}
	Then one easily computes that the spectrum of $T$ is $\sigma(T) = \{1\}$ (hence, the eigenvalue $1$ has algebraic multiplicity $3$) and that the fixed space $\ker(1-T)$ of $T$ is spanned by the vectors $v_1 = (1,0,-1)$ and $v_2 = (0,1,0)$ (hence, the geometric multiplicity of the eigenvalue $1$ is only $2$, which implies that $T$ is not power bounded). We can immediately conclude from the form of $v_1$ and $v_2$ that the positive cone $\bbR^3_+ \cap \ker(1-T)$ is not generating in $\ker(1-T)$, and hence $\ker(1-T)$ cannot be a lattice subspace of $\bbR^3$.
\end{example}

\bibliographystyle{plain}
\bibliography{literature}

\end{document}